\documentclass{amsart}
\usepackage{amssymb,amsmath,comment,mathrsfs,mathtools,tikz-cd,todonotes}
\usepackage[colorlinks=true,linkcolor=blue]{hyperref}
\usepackage{cleveref}
\usepackage{yhmath}

\theoremstyle{remark}

\newtheorem{example}[subsection]{\bf Example}

\newtheorem{rem}[subsection]{\bf Remark}
\newtheorem{convention}[subsection]{\bf Convention}
\theoremstyle{definition}
\newtheorem{notn}[subsection]{\bf Notation}

\newtheorem*{dfn}{Definition}

\theoremstyle{plain}

\newtheorem{theorem}[subsection]{Theorem}

\newtheorem{lemma}[subsection]{Lemma}

\newenvironment{numequation}{\addtocounter{subsection}{1}
\begin{equation}}{\end{equation}}

\newcommand{\bbN}{{\mathbb N}}

\newcommand{\bbQ}{{\mathbb Q}}

\newcommand{\bbZ}{{\mathbb Z}}

\newcommand{\bG}{{\bf G}}

\newcommand{\bL}{{\bf L}}

\newcommand{\bP}{{\bf P}}

\newcommand{\bS}{{\bf S}}
\newcommand{\bT}{{\bf T}}
\newcommand{\bU}{{\bf U}}

\newcommand{\frb}{{\mathfrak b}}

\newcommand{\frg}{{\mathfrak g}}

\newcommand{\frl}{{\mathfrak l}}
\newcommand{\frm}{{\mathfrak m}}
\newcommand{\frn}{{\mathfrak n}}

\newcommand{\frp}{{\mathfrak p}}
\newcommand{\frq}{{\mathfrak q}}

\newcommand{\frs}{{\mathfrak s}}
\newcommand{\frt}{{\mathfrak t}}

\newcommand{\frx}{{\mathfrak x}}

\newcommand{\frz}{{\mathfrak z}}

\newcommand{\cC}{{\mathcal C}}
\newcommand{\cD}{{\mathcal D}}

\newcommand{\cF}{{\mathcal F}}

\newcommand{\cO}{{\mathcal O}}

\newcommand{\cT}{{\mathcal T}}

\newcommand{\alg}{{\rm alg}}

\newcommand{\cl}{{\rm cl}}

\newcommand{\End}{{\rm End}}

\newcommand{\hra}{\hookrightarrow}
\newcommand{\id}{{\rm id}}
\newcommand{\im}{{\rm im}}

\newcommand{\ind}{{\rm ind}}

\newcommand{\lra}{\longrightarrow}
\newcommand{\Max}{{\rm Max}}

\newcommand{\midc}{{\; | \;}}

\newcommand{\ot}{\otimes}

\newcommand{\pr}{{\rm pr}}

\newcommand{\Qp}{{\bbQ_p}}

\newcommand{\ra}{\rightarrow}
\newcommand{\Rep}{{\rm Rep}}

\newcommand{\sm}{{\rm sm}}

\newcommand{\triv}{{\bf 1}}

\newcommand{\Z}{{\mathbb Z}}

\newcommand{\Bcoad}{B\mbox{-}{\rm mod}^{\rm coad}}
\newcommand{\Bcoadfl}{B\mbox{-}{\rm mod}^{\rm coad}_{\rm f.l.}}
\newcommand{\Bmod}{B\mbox{-}{\rm mod}}

\newcommand{\chcFGP}{{\check{\cF}^G_P}}
\newcommand{\chFGQi}{{\check{\cF}^G_{Q_i}}}
\newcommand{\chcFGB}{{\check{\cF}^G_B}}
\newcommand{\chcFGQ}{{\check{\cF}^G_Q}}
\newcommand{\cFGP}{{{\cF}_P^G}}

\newcommand{\cOp}{{\cO^\frp}}

\newcommand{\cTlamu}{\cT_\lambda^\mu}
\newcommand{\cTmula}{\cT_\mu^\lambda}

\newcommand{\DG}{{D(G)}}

\newcommand{\DGcoadfl}{{D(G)}\mbox{-}{\rm mod}^{\rm coad}_{\rm f.l.}}

\newcommand{\DGmod}{{D(G)}\mbox{-}{\rm mod}}
\newcommand{\DGmodlambda}{\DGmod_{|\lambda|}}
\newcommand{\DGmodmu}{\DGmod_{|\mu|}}
\newcommand{\DgB}{D(\mathfrak{g},B)}
\newcommand{\DgP}{D(\mathfrak{g},P)}

\newcommand{\Dlambda}{D(G)\mbox{-}{\rm mod}_{|\lambda|}}

\newcommand{\Dzfin}{D(G)\mbox{-}{\rm mod}_{\frz{\mbox{-}{\rm fin}}}}

\newcommand{\JH}{\mathcal{J}\mathcal{H}}

\newcommand{\Oalglambda}{\cO_{{\rm alg}, |\lambda|}}
\newcommand{\Oalgmu}{\cO_{{\rm alg}, |\mu|}}

\newcommand{\OPalg}{{{\mathcal O}_{\rm alg}^{\mathfrak p}}}

\newcommand{\onu}{\overline{\nu}}
\newcommand{\OGB}{\overline{\cO}_{\rm{alg}}(G,B)}
\newcommand{\OGBlambda}{\OGB_{|\lambda|}}
\newcommand{\OGBmu}{\OGB_{|\mu|}}

\newcommand{\sadm}{{{\rm s}\mbox{-}{\rm adm}}}

\newcommand{\Tlamu}{T_\lambda^\mu}
\newcommand{\Tmula}{T_\mu^\lambda}

\newcommand{\UgE}{{U(\frg_E)}}
\newcommand{\Ugzfin}{U(\frg_E)\mbox{-}{\rm mod}_{\frz{\mbox{-}{\rm fin}}}}
\newcommand{\Uglambda}{U(\frg_E)\mbox{-}{\rm mod}_{|\lambda|}}

\newcommand{\UgEmod}{U(\frg_E)\mbox{-{\rm mod}}}

\newcommand{\Ugfrm}{U(\frg_E)\mbox{-}{\rm mod}_\frm}
\newcommand{\Ugmod}{U(\frg)\mbox{-{\rm mod}}}

\newcommand{\UpE}{{U(\frp_E)}}

\begin{document}

\title{The Functor $\chcFGP$ at the level of $K_0$ }
\author{Akash Jena}
\address{Department of Mathematics, Indiana University, Bloomington}
\curraddr{}
\email{akjena@iu.edu}
\thanks{}

\begin{abstract}
Let $G$ be a $p$-adic Lie group with reductive Lie algebra $\frg$. Denote by $D(G)$ the locally analytic distribution algebra of $G$. Orlik-Strauch and Agrawal-Strauch have studied certain exact  functors defined on various categories of $\frg$-representations with image in the category of locally analytic  $G$-representations or $D(G)$-modules. In this paper we prove that for suitably defined categories of $D(G)$-modules, this functor gives rise to injective homomorphisms at the level of Grothendieck groups. We also explain how this functor interacts with translation functors at the level of Grothendieck groups.
\end{abstract}

\maketitle

\tableofcontents

\section{Introduction}\label{Intro}
Let $F$ be a finite extension of $\Qp$ and consider a split reductive group
$\bf{G}$ over $F$. Let $\bf{P}\supset\bf{B}\supset \bf{T}$ be a parabolic subgroup, a Borel subgroup, and a maximal torus respectively. Denote by $\bL_\bP$ the Levi subgroup of $\bP$ containing $\bT$, and by $\bU_\bP$ the unipotent radical of $\bP$. Let $\frg,\frp, \frb$ and $\frt$ be the Lie algebras of the corresponding algebraic groups. Let $G=\textbf{G}(F)$, $P=\textbf{P}(F)$, $B=\textbf{B}(F)$, $T=\textbf{T}(F)$, $L_P = \bL_\bP(F)$, and $U_P =\bU_\bP(F)$ be the corresponding groups of $F$-valued points.  Furthermore, we fix once and for all a finite extension $E$ of $F$ which will be our field of coefficients. The base change of an $F$-vector space to $E$ will always be denoted by the subscript $E$, for example $\frg_E=\frg\otimes_F E$ and $\frb_E=\frb\otimes_F E$ etc. A weight $\lambda\in\frt^*$ is said to be algebraic if it is in the image of the canonical map $\im(X^*(\bT) \ra \frt^*)$. We denote the set of algebraic weights as $\frt_\alg^*$.

\vskip8pt
We are interested in locally analytic representations of $G$ on $E$-vector spaces, cf. \cite{S-T1,S-T2}. However, instead of working directly with locally analytic representations, we will be working in the framework of $D(G)$-modules where $D(G):=C^{\text{la}}(G,E)_b'$ is the locally analytic distribution algebra, which is the strong dual of the space of $E$-valued locally $F$-analytic functions on $G$, cf. \cite[Cor. 3.3]{S-T1}. 

\vskip 8pt
In the spirit of \cite{AS, O-S2}, we study the functor $\chcFGP$ defined as follows.
\[\begin{array}{ccc}
 \chcFGP   : \OPalg \times \Rep^{\sm}_E(L_P)^\sadm & \lra & \DGmod  \\
     (M,V) & \rightsquigarrow & \DG \ot_{\DgP} (M \ot_E V') \;.
\end{array}\]
Here $\OPalg$ is the full subcategory of the Bernstein-Gelfand-Gelfand category $\cO$ for the Lie algebra $\frg_E$, consisting of modules $M$ with algebraic  weights and on which the action of $\frp_E$ is locally finite. Furthermore,  $\Rep^{\sm}_E(L_P)^\sadm$ is the category of  smooth strongly admissible representations on $E$-vector spaces of the subgroup $L_P$. Here $\DgP$ is the subring of $D(G)$ generated by the universal envelopping algebra $U(\frg_E)$ and $D(P)$. Functors of this form have turned out to be very useful towards various aspects of locally analytic representations and the $p$-adic local Langlands program, see \cite{Breuil,SchraenGL3} for example.

\vskip8pt
The functor $\chcFGP$ is exact in both arguments, cf. \cite[Thm. 4.2.4]{AS}. When the second argument of $\chcFGP$ is the trivial representation $\bf1$ of $P$, we write $\chcFGP(M)$ instead of $\chcFGP(M, \bf1)$. This gives an exact functor $\chcFGP: \OPalg \lra \DGmod$, defined by $M\rightsquigarrow \chcFGP(M)$.   

\vskip 8pt
The image of $\chcFGP$ lies inside the category $\DGcoadfl$, which is the category of coadmissible $D(G)$-modules of finite length,  cf. \ref{image finite length}. We consider the induced map on Grothendieck groups, which we also denote by $\chcFGP$. 
\[\begin{array}{ccc}
 \chcFGP   : K_0(\OPalg) & \lra & K_0(\DGcoadfl)  \\
     \text{[}M\text{]} & \longmapsto & [\chcFGP(M)] \;.
\end{array}\]
Our main result is the following.  
\begin{theorem}[\ref{our main result}]\label{main result intro}
The induced map $\chcFGP:K_0(\OPalg)\ra K_0(\DGcoadfl)$ is injective.
\end{theorem}

Our result follows from an application of \cite[Cor. 2.7]{Breuil} (cf. \ref{Breuil 2.7}) and the irreduciblity result in \cite[5.8]{O-S2}.

\vskip 8pt
We also prove a version of Theorem \ref{main result intro} for various categories indexed by weights (or rather the dot orbits of weights, see \ref{translation subseection}) as explained below. Denote by $\frz_E$ the center of $\UgE$. For $\lambda\in \frt_E^*$, let $\chi_\lambda$ be the central character of $\frz_E$ associated to $\lambda$. Following the notation of \cite{JLS}, we denote by $\Uglambda$ (and $\DGmodlambda$) the full subcategory of $\UgEmod$ (and $\DGmod$) consisting of modules $M$ such that each $m\in M$ is annihilated by sufficiently large powers of $\ker (\chi_\lambda)$. Let $\Oalglambda$ be the full subcategory of $\cO$ consisting of modules with algebraic weights which lie in $\Uglambda$. We show that the image of $\Oalglambda$ under $\chcFGB$ (for a Borel subgroup $B\subset G$) lies in $\OGBlambda$. Here $\OGBlambda$ denotes the smallest full subcategory of $\DGcoadfl$ consisting of modules $M\in \DGmodlambda$ with the following properties:
\vskip 6pt
(i) It contains all modules $\chcFGB(M)$, where $M$ is in $\cO_\alg$. 

\vskip 4pt

(ii) It is closed under taking subquotients and extensions.
\vskip 8pt
With this notation in place we have the following result which is a consequence of \ref{main result intro}.

\vskip8pt 
\begin{theorem}[\ref{main result v2}] \label{main result v2 intro}
The induced map $\chcFGB:K_0(\Oalglambda)\ra K_0(\OGBlambda)$ is injective.
\end{theorem}

As an application of Theorem \ref{main result v2 intro}, in Section \ref{Section on translation}, we explain how the functor $\chcFGB$ interacts with translation functors at the level of Grothendieck groups. We consider the translation functors $\Tlamu$ and $\cTlamu$ (cf. \cite[2.4.1]{JLS} and \cite[2.4.5]{JLS}). In that direction we have the following result. 

\vskip 8pt

\begin{theorem}[\ref{commutative theorem}]
Let $\lambda, \mu\in \frt_E^*$ be such that $\mu-\lambda$ is algebraic. Then the functors $\chcFGB$, $\Tlamu$ and  $\cTlamu$ give rise to the following commutative diagram of abelian groups.
\[
\begin{tikzcd}
K_0(\Oalglambda) \arrow[d, "\Tlamu"] \arrow[rr, "\chcFGB", hook] & & K_0\left(\OGBlambda\right)\arrow[d, "\cTlamu"]\\
K_0(\Oalgmu) \arrow[rr, "\chcFGB", hook] & & K_0\left(\OGBmu\right)
\end{tikzcd}
\]
\vskip 8pt
where the horizontal maps are injective. Additionally, if $\lambda,\mu$ satisfy the conditions (ii) and (iii) in \ref{categirical equivalence}, then the vertical maps are isomorphisms.
\end{theorem}

\textbf{Acknowledgement} The author is grateful to his advisor Prof\text{.} Matthias Strauch for his invaluable advice during many discussion sessions throughout this project.

\section{Jordan-H\"older series for coadmissible modules}
In this section we consider a Fr\'echet-Stein algebra $A$ over $E$ and a ring extension $A \subset B$, where $B$ is an associative unital $E$-algebra. We call a $B$-module coadmissible, if it coadmissible as an $A$-module. Recall that a coadmissible $A$-module  carries a canonical topology \cite[before 3.6]{S-T2}.
\begin{dfn}
Let $M$ be a coadmissible $B$-module.

\vskip8pt

(i) We say that a $B$-module $M$ is {\it topologically simple} if it has no  closed submodule other than $0$ and $M$.

\vskip8pt

(ii)  A {\it topological composition series} of $M$ is a finite chain of closed $B$-submodules 
\[0=M_0\subsetneq M_1\subsetneq \dots \subsetneq M_n=M\]
such that $M_i/M_{i-1}$ is topologically simple for $1\leq i\leq n$.
\end{dfn}

\begin{convention} From now on, if not said otherwise, we will call a topological composition series simply a composition series. This should not lead to confusion with the concept of a composition series in the sense of the theory of (abstract) modules over a ring.
\end{convention}

\begin{lemma}\label{joho submodule}
Let $M$ be a coadmissible $B$-module. If $M$ has a composition series, then so does any closed submodule $N\subset M$. 
\end{lemma}
\begin{proof}
Take any composition series for $M$, say
\[0=M_0\subsetneq M_1 \subsetneq \dots \subsetneq M_n=M\]
Taking intersections with the submodule $N$ yields a chain of closed $B$-submodules of $N$,
\begin{numequation}\label{joho sub proof}
0=M_0\cap N\subseteq M_1\cap N\subseteq \dots \subseteq M_n\cap N=N.
\end{numequation}
It is possible for the submodules in  the chain in \ref{joho sub proof} to be equal, so it may not be a composition series. We have 

\begin{numequation}\label{joho sub proof 2}
\begin{array}{rcl} 
(M_i\cap N) /(M_{i-1}\cap N) &=& (M_i\cap N) /(M_{i-1}\cap(M_{i}\cap N))\\
&&\\
& \cong & ((M_i\cap N)+M_{i-1})/ M_{i-1}\\
&&\\
& \subseteq & M_{i}/M_{i-1}.
\end{array}
\end{numequation}
Since $M_i/M_{i-1}$ is simple, the modules $(M_i\cap N) /(M_{i-1}\cap N)$ occurring in \ref{joho sub proof 2} are simple or zero. 
\end{proof}

\begin{dfn}
Let $M$ be a coadmissible $B$-module. Consider the following two composition series for $M$
\begin{numequation}\label{equiv def 1}
0=M_0\subsetneq M_1 \dots \subsetneq M_n=M
\end{numequation}

\begin{numequation}\label{equiv def 2}
0=M'_0\subsetneq M'_1 \dots \subsetneq M'_m=M
\end{numequation}
The composition series in \ref{equiv def 1} and \ref{equiv def 2} are said to be \it{equivalent} if the following conditions hold.
\begin{itemize}
    \item[(i)] $m=n$.
    \item[(ii)] There is a permutation $\sigma$ of $\{1,\dots, n\}$ such that $M_i/M_{i-1}\cong M'_{\sigma(i)}/M'_{\sigma(i)-1}$
\end{itemize}
\end{dfn}

\vskip8pt

The goal of this section is to prove that any two composition series of a coadmissible $D$-module are equivalent. We closely follow the arguments for abstract $B$-modules given in \cite[3.2]{EH}.

\vskip8pt

\begin{lemma}\label{intersection lemma}
Let $M$ be a coadmissible $B$-module. Consider the following two composition series
\begin{numequation}\label{last term 1}
0=M_0\subsetneq M_1 \dots \subsetneq M_n=M
\end{numequation}

\begin{numequation}\label{last term 2}
0=M'_0\subsetneq M'_1 \dots \subsetneq M'_m=M.
\end{numequation}
Suppose $M'_{m-1}\neq M_{n-1}$ and consider $S= M'_{m-1}\cap M_{n-1}$. Then the inclusion maps $M_{n-1} \hra M$ and $M'_{m-1} \hra M$ induce isomorphisms 
\[M_{n-1}/S\cong M/M'_{m-1}\hskip10pt \mbox{ and  } \hskip10pt  M'_{m-1}/S \cong M/M_{n-1}\]
and hence both of these quotients are topologically simple.
\end{lemma}
\begin{proof}
Note first that $N := M_{n-1}+M'_{m-1}$ is again a coadmissible $A$-module, by \cite[3.4 (iii)]{S-T2}, and hence closed by \cite[3.6]{S-T2}. If $N$ would be equal to $M_{n-1}$, then $M'_{m-1}$ would be contained in $M_{n-1}$, and hence would be properly contained in $M_{n-1}$, since we assume that $M_{n-1} \neq M'_{m-1}$. This shows that $N$ properly contains $M_{n-1}$, and thus $N = M$ since $M_{n-1}$ is maximal among the closed prper $B$-submodules of $M$. 

\vskip8pt

Now we have 
\[M/M'_{m-1}=(M_{n-1}+M'_{m-1})/M'_{m-1}\cong M_{n-1}/(M_{n-1}\cap M'_{m-1}) = M_{n-1}/S.\]
Similarly, we show that $M'_{m-1}/S \cong M/M_{n-1}$.
\end{proof}

\begin{theorem}
Let $M$ be a coadmissible $B$-module. Consider any two composition series for $M$
\begin{numequation}\label{equiv ser 1}
0=M_0\subsetneq M_1 \dots \subsetneq M_n=M
\end{numequation}

\begin{numequation}\label{equiv ser 2}
0=M'_0\subsetneq M'_1 \dots \subsetneq M'_m=M.
\end{numequation}
Both of these series are equivalent.
\end{theorem}
\begin{proof}
We will proceed by induction on $n$, which is the length of the chain in \ref{equiv ser 1}. If $n=0$, then $M=0$ and there is nothing to prove. 
%If $n=1$, then $M$ is topologically simple and so $M'_1=M$. So our claim is true for $n \le 1$.
\vskip 8pt
Now suppose that $n>0$. By the induction hypothesis, the theorem holds for modules which have composition series of length $\leq n-1$.
\vskip 8pt
First let us consider the case when $M_{n-1}=M'_{m-1}=: T$ say. Then $T$ inherits a composition series of length $n-1$ from \ref{equiv ser 1}. By the induction hypothesis, any two composition series for $T$ have the same length. Thus the composition series inherited from \ref{equiv ser 2} is also of length $n-1$. Hence $m-1=n-1$ and thus $m=n$. By the induction hypothesis we also have a permutation $\sigma$ of $\{1,\dots, n-1\}$ such that $M_{i}/M_{i-1}\cong M'_{\sigma(i)}/M'_{\sigma(i-1)}$. We also have $M_n/M_{n-1} = M/T = M'_n/M'_{n-1}$. So if we view $\sigma$ as a permutation of $\{1,\dots, n\}$ fixing $n$, we have the required permutation.
\vskip 8pt
Now assume that $M_{n-1}\neq M'_{m-1}$. We define $S:=M_{n-1}\cap M'_{m-1}$. Consider a composition series of $S$, which exists by \ref{joho submodule}
\[0=S_0\subsetneq S_1\subsetneq \dots \subsetneq S_t= S.\]
By \ref{intersection lemma}, the modules $M_{n-1}/S$ and $M'_{m-1}/S$ are topologically simple. So we get the following two composition series of $M$
\begin{numequation}\label{joho ser intersection 1}
0=S_0\subsetneq S_1\subsetneq \dots \subsetneq S_t= S\subsetneq M_{n-1}\subsetneq M
\end{numequation}
and
\begin{numequation}\label{joho ser intersection 2}
0=S_0\subsetneq S_1\subsetneq \dots \subsetneq S_t= S\subsetneq M'_{m-1}\subsetneq M.
\end{numequation}
The composition factors appearing in both \ref{joho ser intersection 1} and \ref{joho ser intersection 2} are the same till the the module $S$. But the last two composition factors are also going to be same, thanks to \ref{intersection lemma}. So both the composition series \ref{joho ser intersection 1} and \ref{joho ser intersection 2} are equivalent.

\vskip 8pt

Now we claim that $m=n$. The module $M_{n-1}$ inherits a composition series of length $n-1$ from \ref{equiv ser 1}. So by the induction hypothesis all composition series of $M_{n-1}$ have length $n-1$. The composition series inherited from \ref{joho ser intersection 1} has length $t+1$ and hence $t+1=n-1$. Similarly the module $M'_{m-1}$ inherits a composition series of length $m-1$ from \ref{equiv ser 2}. Again, by induction hypothesis, any two composition series of $M'_{m-1}$ have the same length, and so $t+1=m-1$ from \ref{joho ser intersection 2}. Therefore $m=n$.

\vskip 8pt

Next we show that the composition series \ref{equiv ser 1} and \ref{joho ser intersection 1} are equivalent. By the induction hypothesis the composition series of $M_{n-1}$ inherited from \ref{equiv ser 1} and \ref{joho ser intersection 1} are equivalent. So there is a permutation $\gamma$ of $\{1, \dots, n-1\}$ such that
\[S_i/S_{i-1}\cong M_{\gamma(i)}/M_{\gamma(i-1)}, (i\neq n-1) \text{ and }M_{n-1}/S\cong M_{\gamma(n-1)}/M_{\gamma(n-1)-1}\]
We view $\gamma$ as a permutation of $\{1,\dots, n\}$ which fixes $n$. Then
\[M/M_{n-1}=M_n/M_{n-1}\cong M_{\gamma(n)}/M_{\gamma(n)-1}\]
which proves that \ref{equiv ser 1} and \ref{joho ser intersection 1} are equivalent. Similarly we can show that \ref{equiv ser 2} and \ref{joho ser intersection 2} are equivalent. Thus it follows that \ref{equiv ser 1} and \ref{equiv ser 2} are equivalent.
\end{proof}

\vskip8pt

If a coadmissible $B$-module has a topological composition series of length $n$, then we call $n$ the {\it topological length} of $M$. We will usually just say ``length'' instead of ``topological length''.

\begin{lemma}\label{joho quotient}
Let $M$ be a coadmissible $B$-module of finite length and $N\subset M$ be a coadmissible $B$-submodule. Then $M/N$ is a coadmissible $B$-module of finite length.
\end{lemma}
\begin{proof}
By \cite[3.6]{S-T2}, $M/N$ is coadmissible as an $A$-module and hence coadmissible as a $B$-module. We construct a composition series of $M/N$ from a composition series of $M$. Consider the following composition series of  $M$
\begin{numequation}\label{joho quotient proof 1}
0=M_0\subsetneq M_1\subsetneq \dots \subsetneq M_n=M.
\end{numequation}
Then we have a series of submodules of $M/N$
\begin{numequation}\label{joho quotient proof 2}
0=(M_0+N)/N\subseteq (M_1+N)/N\subseteq \dots \subsetneq (M_n+N)/N=M/N.
\end{numequation}
It is possible that some of  the terms in \ref{joho quotient proof 2} are equal. Each $M_i+N$ is coadmissible as an $A$-module by \cite[3.4 (iii)]{S-T2}, and hence closed by \cite[3.6]{S-T2}. The quotient $(M_i+N)/N$ is coadmissible by \cite[3.6]{S-T2}.  We have
\begin{numequation}\label{joho quotient proof 3}
\begin{array}{rcl} 
((M_i+ N)/N) /((M_{i-1}+ N)/N) &\cong& (M_i+ N) /(M_{i-1}+ N)\\
&&\\
& = & (M_{i-1}+ N+M_i) /(M_{i-1}+ N)\\
&&\\
& \cong & M_i/((M_{i-1}+ N)\cap M_i)
\end{array}
\end{numequation}
where the equality in the second step holds because $M_{i-1}\subset M_i$. We also have $M_{i-1}\subset M_{i-1}+N$ and therefore $M_{i-1}\subset(M_{i-1}+N)\cap M_i\subset M_i$. But $M_{i-1}$ is a maximal closed submodule of $M_i$. Hence the quotient $M_i/((M_{i-1}+ N)\cap M_i)$ is simple or zero. If we ommit all the terms that are zero in \ref{joho quotient proof 2}, we end up with a composition series of $M/N$. 
\end{proof}

\vskip8pt

\section{Coadmissible modules of finite length}\label{section fl}
Define $\Bcoad$ to be the full subcategory of the category $\Bmod$ of all left $B$-modules which consists of the coadmissible modules. Furthermore, we denote by $\Bcoadfl$ the full subcategory of $\Bcoad$ consisting of modules of finite length.

\begin{lemma}\label{abelian lemma} 
(i) $\Bcoad$ is an abelian category.

\vskip8pt

(ii) $\Bcoadfl$ is an abelian category. 
\end{lemma}

\begin{proof}
%\todo{find a reference, and/or elaborate, why this is true} 
(i) Consider a morphism $f:M\ra N$ of coadmissible $B$-modules. Then $f$ is also a morphism of coadmissible $A$-modules. The kernel and cokernel of $f$ as a morphism of $A$-modules coincides with the respective kernel and cokernel as a morphism of $B$-modules. The category of coadmissible $A$-modules is abelian by \cite[3.5]{S-T2} and hence $\Bcoad$ is also abelian.

\vskip 8pt

(ii) $\Bcoadfl$ is a full subcategory of $\Bcoad$. Consider any $M\in \Bcoadfl$ and a coadmissible submodule $N\subset M$. Since $\Bcoad$ is abelian, it is enough to show that $N\in \Bcoadfl$ and $M/N\in \Bcoadfl$.  But that follows from \ref{joho submodule} and \ref{joho quotient} respectively. 
\end{proof}

\begin{dfn}
As defined in \cite[Appendix E]{Prest} we say that a category $\cC$ is \textit{skeletally small} if the isomorphism classes of objects in $\cC$ is a set.  
\end{dfn}

\begin{lemma}\label{Skeletally small lemma}
(i) For any ring $R$, the category of finitely generated $R$-modules is skeletally small.

\vskip 8pt

(ii) The category $\Bcoad$ is skeletally small.

\vskip 8pt

(iii) The category $\Bcoadfl$ is skeletally small.
\end{lemma}
\begin{proof}
(i) Let $Q^{\text{fg}}(R^{\bbN})$ be the set of quotients $N$ of the countably
free $R$-module $R^{\bbN}$, and which have the property that the quotient map $R^{\bbN} \ra N$ factors through some $R^m$ (which is a quotient of $R^{\bbN}$ by sending all components indexed by $k>m$ to zero). Any such $N$ is finitely generated. It is clear that $Q^{\text{fg}}(R^{\bbN})$ is a set and any finitely generated module is isomorphic to one in $Q^{\text{fg}}(R^{\bbN})$. Let $I^{\text{fg}}(R)$ be the set of isomorphism classes in $Q^{\text{fg}}(R^{\bbN})$. So every finitely generated module is isomorphic
to a module in $Q^{\text{fg}}(R^{\bbN})$ which in turn determines a unique element in $I^{\text{fg}}(R)$.

\vskip 8pt
(ii) Consider a coadmissible $B$-module $M$. Write $A$ as the projective limit 
\[A\cong \varprojlim_n A_n\]
of a countable projective system $(A_n)_{n \ge 0}$ of (left) Noetherian Banach algebras $A_n$, where for each $n \ge 0$ the ring $A_n$ is a flat right $A_{n+1}$-module via the transition map $A_{n+1} \ra A_n$, cf. \cite[sec. 3]{S-T2}. We consider the finitely generated $A_n$-module $M_n$ defined as follows
\[M_n\cong A_n\otimes_A M \;,\] 
cf. \cite[Cor. 3.1]{S-T2}. Let $I^{\rm fg}(A_n)$ be the collection of isomorphism classes of finitely generated $A_n$-modules. Then $I^{\rm fg}(A_n)$ is a set by part (i). Given two isomorphic coadmissible $B$-modules $M$ and $N$, we have $M_n\cong N_n$ for all $n$. Thus we have the equality ${\rm cl}_{A_n}(M_n) = {\rm cl}_{A_n}(N_n)$ in $I^{\rm fg}(A_n)$, where ${\rm cl}_{A_n}(M_n)$ and ${\rm cl}_{A_n}(N_n)$ are the isomorphism classes of $M_n$ and $N_n$, respectively. It follows that if $M$ and $N$ are isomorphic coadmissible $B$-modules, then the sequences $({\rm cl}_{A_n}(M_n))_n$ and $({\rm cl}_{A_n}(N_n))_n$ are identical. It follows that the collection of isomorphism classes of coadmissible $B$-modules is a set. 

\vskip 8pt

(iii) Since $\Bcoad$ is skeletally small, any subcategory of $\Bcoad$ is also skeletally small. Hence $\Bcoadfl$ is skeletally small.
\end{proof}

%\todo{introduce notatino for the isomorphism class of a simple object}

\begin{notn}
Given any $M\in \Bcoadfl$, we denote by $[M]$ its class in the Groethendieck group $K_0(\Bcoadfl)$.
\end{notn}

\begin{lemma}\label{Basis of K_0}
The Grothendieck group $K_0(\Bcoadfl)$ is a free abelian group with basis $\{[M] \text{ }| M \emph{ is simple}\}$.
\end{lemma}
\begin{proof}
This follows from \ref{Skeletally small lemma} and \cite[Theorem C]{Enomoto}.
\end{proof}

\begin{rem}
We do not know if a topologically simple coadmissible $B$-module is monogenic in the algebraic sense. Similarly, we also do not know if a coadmissible $B$-module of topologically finite length is finitely generated in the algebraic sense. 
\end{rem}

\vskip 16pt

\section{The map induced by the functor $\chcFGP$ on the Grothendieck groups}
We recall our notation from Section \ref{Intro}. Let $F$ be a finite extension of $\Qp$ and $\bf{G}$ be a split reductive group
over $F$. Let $\bf{P}\supset\bf{B}\supset \bf{T}$ be a parabolic subgroup, a Borel subgroup, and a maximal torus respectively. Denote by $\frg,\frp, \frb$ and $\frt$ their respective Lie algebras. Let $G=\textbf{G}(F)$, $P=\textbf{P}(F)$, $B=\textbf{B}(F)$ and $T=\textbf{T}(F)$ be the corresponding groups of $F$-valued points.  We fix once and for all a finite extension $E$ of $F$ which will be our field of coefficients. The base change of an $F$-vector space to $E$ will always be denoted by the subscript $E$, for example $\frg_E=\frg\otimes_F E$ and $\frb_E=\frb\otimes_F E$ etc. In this section closely follow the notation of \cite{O-S2, AS}.
\subsection{Category \texorpdfstring{$\OPalg$}{} and the functor \texorpdfstring{$\chcFGP$}{}}
Given a representation $\phi: \frt_E \ra \End_E(M)$, a weight $\lambda \in \frt^*_E$, and a positive integer $i \ge 1$ we set 

\[M_\lambda = \{m \in M \midc \forall \frx \in \frt_E: (\phi(\frx)-\lambda(\frx)\cdot \id).m = 0\} \;.\]

\vskip8pt

\vskip8pt

The category $\cO^\frp$ for the pair $(\frg,\frp)$ and the coefficient field $E$ is defined to be the full subcategory of all $\UgE$-modules $M$ which satisfy the following properties:

\vskip8pt

\begin{enumerate}
    \item  $M$ is finitely generated as a $\UgE$-module.
    \item  $M = \bigoplus_{\lambda \in \frt^*_E} M_\lambda$ \;. 
    \item  The action of $\frp_E$ on $M$ is locally finite, i.e. for every $m\in M$, the subspace $\UpE .m \subset M$ is finite-dimensional over $E$.
\end{enumerate}

\vskip8pt

When $\frp=\frb$, the category $\cO^\frp$ coincides with the Bernstein-Gelfand-Gelfand category $\cO$. We denote by $\OPalg$ the full subcategory $\cOp$ consisting of modules $M$ for which the weights lie in the image $\frt^*_\alg$ of the differential $X^*(\textbf{T})\ra \frt^*$. 

\vskip8pt
Let $D(G):=C^{\text{la}}(G,E)_b'$ be the locally analytic distribution algebra, which is the strong dual of the space of $E$-valued locally $F$-analytic functions on $G$, cf. \cite[Cor. 3.3]{S-T1}. Denote by $\DgP$ the subring of $D(G)$ generated by the universal envelopping algebra $U(\frg_E)$ and $D(P)$.

\begin{dfn}
Denote by $\Rep^{\sm}_E(L_P)^\sadm$ the category of smooth strongly admissible representations on $E$-vector spaces of the subgroup $L_P$. In the spirit of \cite{AS} we consider the functor
\[\begin{array}{ccc}
 \chcFGP   : \OPalg \times \Rep^{\sm}_E(L_P)^\sadm & \lra & \DGmod  \\
     (M,V) & \rightsquigarrow & \DG \ot_{\DgP} (M \ot_E V') \;,
\end{array}\] cf. \cite[4.2.1]{AS} (where ${\rm Lift}(M,\log)$ is the canonical lift of $M$ to a module over $D(\frg,P)$, as explained in \cite[3.2, 3.6]{O-S2}).
\end{dfn}

\vskip 8pt

Let $\Rep^{\rm la}_E(G)$ be the category of locally analytic $G$-representations on $E$-vector spaces.  We would like to point out that there is an analogue of $\chcFGP$, with image in $\Rep^{\rm la}_E(G)$ rather than $\DGmod$ as described in \cite{O-S2, O-S3}. This functor is denoted by $\cF^G_P$. The functors $\chcFGP$ and $\cFGP$ are related in the following way
\[\cFGP(M,V)=\chcFGP(M,V)'_b\]
for $M\in \OPalg$ and $V\in \Rep^{\sm}_E(L_P)^\sadm$. Here $\chcFGP(M,V)'_b$ denotes the dual of $\chcFGP(M,V)$ equipped with the strong topology.

\vskip 8pt

\subsection{$\chcFGP$ at the level of $K_0$}
In accordance with the notation introduced in Section \ref{section fl}, we will study the category $\Bcoadfl$ where $B=D(G)$. So the category $\DGcoadfl$ is abelian by \ref{abelian lemma}(ii) and skeletally small by \ref{Skeletally small lemma}(iii).

\begin{lemma}\label{image finite length}
$\chcFGP(M)\in \DGcoadfl$ for any $M\in \cO$. \end{lemma}
\begin{proof}
Coadmissibilty holds by \cite[4.2.3]{AS}. Since $M$ has finite length, and because $\chcFGP$ is an exact functor, it suffices to show that $\chcFGP(M)$
has finite length when $M$ is simple as a module over the enveloping algebra. If $\frq$ is maximal for $M$ in the sense of \cite[5.2]{O-S2}, then we have $\chcFGP(M) = \chcFGQ(M,\ind^{L_Q}_{L_P(L_Q \cap U_P)}(\triv))$ by \cite[4.3.3]{AS}, where $Q$ is the standard parabolic subgroup with Lie algebra $\frq$. Since $\ind^{L_Q}_{L_P(L_Q \cap U_P)}(\triv)$ has finite length, it suffices to show that $\chcFGQ(M,\pi)$ is topologically simple as $D(G)$-module, when $\pi$ is an irreducible smooth $L_P$-representation. This is a consequence of $\cF^G_Q(M,\pi)$ being topologically irreducible, by \cite[5.8]{O-S2}. 
\end{proof}

\begin{comment}
\begin{convention}
Henceforth, for any $M\in \DGcoadfl $we will use the term ``Jordan-H\"older series of $M$" and $\JH(M)$ synonymously.
\end{convention}
\end{comment}

\begin{notn}
Given an exact functor $F:\cC\ra \cD$ between two abelian categories, we get an induced homomorphism between Grothendieck groups
\begin{align*}
    &K_0(\cC)\ra K_0(\cD)\\
    &[M]\mapsto [FM].
\end{align*}
By abuse of notation we will also denote this map as $F:K_0(\cC)\ra K_0(\cD)$.
\end{notn}

We recall the following notion of maximality from \cite{O-S2}. For $M\in \cO^\frp$, we say that the parabolic subalgebra $\frp$ is \textit{maximal} for $M$ if $M$ does not lie in $\cO^\frq$ for any parabolic subalgebra $\frq$ properly containing $\frp$.  A crucial input in the proof of our main result is the following theorem of C. Breuil \cite[Cor. 2.7]{Breuil}.

\begin{theorem}\label{Breuil 2.7}
Let $M_1$ and $M_2$ be two simple objects of $\OPalg$, $Q_1$ and $Q_2$ be their maximal parabolics, $\pi_{Q_1}$ and $\pi_{Q_2}$ be two smooth admissible representations of finite length of $L_{Q_1}(F)$ and $L_{Q_2}(F)$ respectively on $E$. Then we have $\cF_{Q_1}^G(M_1, \pi_{Q_1})\cong \cF_{Q_2}^G(M_2, \pi_{Q_2})$ if and only if $Q_1 = Q_2$ and $M_1\cong M_2$ and $\pi_{Q_1}\cong \pi_{Q_2}$.
\end{theorem}

\vskip8pt
Consider $\lambda\in \frt_E^*$. Following \cite{Hu}, we will denote by $L(\lambda)$ the simple highest weight module of weight $\lambda$ in $\cO$. Thus $L(\lambda)$ is the unique irreducible quotient of the Verma module $U(\frg_E)\otimes_{U(\frb_E)} \textbf{1}_{\lambda}$. Now we present our main result.

\begin{theorem}\label{our main result}
The induced map $\chcFGP:K_0(\OPalg)\ra K_0(\DGcoadfl)$ is injective.  
\end{theorem}
\begin{proof}
The set $\{[L(\lambda)]\text{ }| \text{ }L(\lambda)\in \OPalg\}$  forms a $\bbZ$-basis of $K_0(\OPalg)$, cf. \cite[Sec. 1.11]{Hu}. So any arbitrary element $\xi \in K_0(\cO)$ is of the form
\[\xi=\sum_{i=1}^{n}m_i[L(\lambda_i)]\]
\vskip 4pt
where $L(\lambda_i)$ are pairwise nonisomorphic simple modules in $\cO$ and $m_i\in \bbZ$. We need to show that if $\chcFGP(\xi)=0$, then $\xi=0$.  So assume that $\chcFGP(\xi)=0$. Then 
\begin{numequation}
0=\chcFGP\left(\sum_{i=1}^nm_i[L(\lambda_i)]\right)=\sum_{i=1}^nm_i\chcFGP\left([L(\lambda_i)]\right)
\end{numequation}
\vskip 4pt
It follows that 
\begin{numequation}
\sum_{m_i>0}m_i\chcFGP\left([L(\lambda_i)]\right)=\sum_{m_i<0}-m_i\chcFGP\left([L(\lambda_i)]\right)
\end{numequation}
\vskip 4pt
which is equivalent to
\begin{numequation}\label{positive and negative m's}
\sum_{m_i>0}m_i\left[\chcFGP(L(\lambda_i))\right]=\sum_{m_i<0}-m_i\left[\chcFGP(L(\lambda_i))\right].
\end{numequation}
\vskip 4pt
Let $\frq_i$ be the standard parabolic subalgebra which is maximal for $L(\lambda_i)$ in the sense of \cite{O-S2}. Denote by $Q_i$ the subgroup which corresponds to $\frq_i$. Now we have 
\begin{numequation}\label{PQ formula}
\chcFGP(L(\lambda_i)) \; = \; \chcFGP(L(\lambda_i), \mathbf{1}) \; \cong \;  \chFGQi(L(\lambda_i), \text{ind}^{Q_i}_B(\mathbf{1})) \; \cong \; \chFGQi(L(\lambda_i), \text{ind}^{L_{Q_i}}_{B\cap L_{Q_i}}(\mathbf{1})) \;.
\begin{comment}
\begin{array}{rcl}\label{PQ formula}
\chcFGP(L(\lambda_i)) &=& \chcFGP(L(\lambda_i), \mathbf{1})\\
&&\\
& \cong& \chFGQi(L(\lambda_i), \text{ind}^{Q_i}_B(\mathbf{1}))\\
&&\\
& \cong & \chFGQi(L(\lambda_i), \text{ind}^{L_{Q_i}}_{B\cap L_{Q_i}}(\mathbf{1})).
\end{array}    
\end{comment}
\end{numequation}
\vskip 4pt
where the isomorphism in the second step follows from \cite[4.3.3]{AS}. Let $\pi_{i,1}, \pi_{i,2}, \dots, \pi_{i, k_i}$ be the distinct Jordan-H\"older factors of $\text{ind}^{L_{P_i}}_{B\cap L_{P_i}}(\mathbf{1})$ with multiplicities $\mu_{i,1}, \mu_{i,2}, \dots, \mu_{i,k_i}$ respectively. Since $\chFGQi$ is exact, it follows that
\begin{numequation}\label{FPG simplicity OS2}
\begin{array}{rcl} 
\left[\chFGQi(L(\lambda_i), \text{ind}^{L_{Q_i}}_{B\cap L_{Q_i}}(\mathbf{1}))\right] &=& \sum_{j=1}^{k_i} \mu_{i,j} \cdot  \left[\chFGQi(L(\lambda_i), \pi_{i,j})\right]\\
&&\\
\end{array}
\end{numequation}
\vskip 4pt
Combining \ref{PQ formula} and \ref{FPG simplicity OS2} we conclude that 
\begin{numequation}\label{JH of F(L(lambda))}
\left[\chcFGP(L(\lambda_i))\right]=\sum_{j=1}^{k_i} \mu_{i,j}\cdot \text{ }\left[\chFGQi(L(\lambda_i), \pi_{i,j}\right].
\end{numequation}
Now it follows from \ref{positive and negative m's} and \ref{JH of F(L(lambda))}
\begin{numequation}\label{positive and negative multiset}
\sum_{m_i>0} \left(\sum_{j=1}^{k_i} m_i\mu_{i,j} \cdot \left[\chFGQi(L(\lambda_i), \pi_{i,j}\right]\right)=\sum_{m_i<0} \left(\sum_{j=1}^{k_i} -m_i\mu_{i,j} \cdot \left[\chFGQi(L(\lambda_i), \pi_{i,j})\right]\right).
\end{numequation}
\vskip 8pt
Recall that the set $S=\{[M]\text{ }|\text{ } M\in \DGcoadfl \text{ is simple}\}$ forms a basis for $K_0(\DGcoadfl)$ by \ref{Basis of K_0}. By \cite[5.8]{O-S2}, the representations $\chFGQi(L(\lambda_i), \pi_{i,j})$ appearing in \ref{positive and negative multiset} are all irreducible. So the terms $\left[\chFGQi(L(\lambda_i), \pi_{i,j})\right]$ are all elements of the basis $S$. Furthermore, since the $L(\lambda_i)$'s are pairwise nonisomorphic, the terms $\left[\chFGQi(L(\lambda_i), \pi_{i,j})\right]$ are all pairwise distinct by \ref{Breuil 2.7}. So the equality in \ref{positive and negative multiset} holds only if the product $m_i\mu_{i, j}=0$ for all $i,j$. But $\mu_{i,j}$ is nonzero by definition. Hence we must have $m_i=0$ for all $i$. Thus $\xi =\sum_{i=1}m_i[L(\lambda_i)]=0$.
\end{proof}

\section{Relation with translation functors}\label{Section on translation}
Our goal in this section is to explain how the functor $\chcFGB$ (for a fixed Borel subgroup $B\subset G$) interacts with translation functors at the level of Grothendieck groups (see \ref{commutative theorem}). We give a brief overview of translation functors in \ref{translation subseection}. Translation functors for Lie algebra representations have been studied in \cite{BeGe, JantzenModuln, Hu}. Translation functors for locally analytic representations were introduced in \cite{JLS}. We closely follow the notation of \cite{JLS}.

\subsection{Translation functors}\label{translation subseection}
Let $\frz_E$ be the center of $U(\frg_E)$ and let $\Max(\frz_E)$ be the set of maximal ideals of $\frz_E$. Denote by $\Ugzfin$ the full subcategory of $\UgE$-mod consisting of modules $M$ such that each $m \in M$ is annihilated by an ideal of finite codimension in $\frz_E$. Modules $M$ in this category have a direct sum decomposition of the form 
\[M = \bigoplus_{\frm \in \Max(\frz_E)} M_\frm\]
where $M_\frm$ is the submodule of $M$ consisting of elements which are annihilated by powers of $\frm$. Denote by $\Ugfrm$ the subcategory of modules $M$ with the property that $M = M_\frm$. The projection $M \ra M_\frm$ is denoted by $\pr_\frm$. Translation functors are endo-functors of $\Ugzfin$ of the form

\[M \rightsquigarrow \pr_\frn (L \ot_E \pr_\frm(M)) \;,\]

\vskip8pt
where $L$ is a finite-dimensional irreducible representation of $\frg_E$, and $\frm,\frn \in \Max(\frz_E)$. 

\vskip8pt
For a weight $\lambda\in\frt_E^*$ let  $\chi_{\lambda}$ be the central character of $\frz_E$ associated to $\lambda$. Let $|\lambda|$ be the orbit of $\lambda$ under the dot action of the Weyl group $W$ of $(\bG,\bT)$. Consider the maximal ideal $\ker(\chi_\lambda)$. In accordance with the notation introduced in \cite{JLS}, we write $\pr_{|\lambda|}$ and $\Uglambda$ instead of $\pr_{\ker(\chi_\lambda)}$ and $\Ugmod_{\ker(\chi_\lambda)}$, respectively. Notice that $\pr_{|\lambda|}=\pr_{|\theta|}$ and $\UgE\text{-mod}_{|\lambda|}=\UgE\text{-mod}_{|\theta|}$ for any $\theta\in |\lambda|$ by the Harish-Chandra Theorem for reductive Lie algebras, cf. \cite[4.115]{KnappVogan_Cohomological}.

\vskip 8pt
\begin{dfn}
Given $\lambda,\mu \in \frt^*_E$ such that $\nu := \mu -\lambda$ is integral, the translation functor $\Tlamu$ is the defined by

\begin{numequation}\label{intro T functor}
\begin{array}{lccc}\Tlamu: & \Ugzfin & \ra & \Ugzfin\\
&&&\\
& M & \rightsquigarrow & \pr_{|\mu|} (L(\onu) \ot_E \pr_{|\lambda|}(M)) \;,
\end{array}
\end{numequation}

\vskip8pt

where $\onu$ is the dominant weight in the (linear) Weyl orbit of $\nu$ and $L(\onu)$ the irreducible finite-dimensional module with highest weight $\onu$.
\end{dfn}

Now we review the construction of translation functors for $D(G)$-modules. There is a canonical injective algebra homomorphism $\UgE \hra \DG$ under which $\frz_E$ is mapped into the center of $D(G)$. As a consequence, the categories $\Dzfin$ and $\Dlambda$ can be readily defined.

\vskip8pt
\begin{dfn}
The translation functor 

\begin{numequation}\label{intro cT functor}
\begin{array}{lccc}\cTlamu: & \Dzfin & \ra & \Dzfin\\
&&&\\
& M & \rightsquigarrow & \pr_{|\mu|} (L(\onu) \ot_E \pr_{|\lambda|}(M))
\end{array}
\end{numequation}

\vskip8pt

is then defined exactly as in (\ref{intro T functor}), except that we require that $\onu$ lifts to an algebraic character of $\bT$, which in turn ensures that $L(\onu)$ lifts to an algebraic representation of $\bG$.
\end{dfn}

\vskip 8pt
\subsection{Induced maps on Grothendieck groups}
We denote by $\cO_\alg$ the full subcategory of $\cO$ consisting of modules with algebraic weights. 
\begin{dfn}
We define $\OGB$ to be the smallest full subcategory of $\DGcoadfl$ which has the following properties.
\vskip 4pt
(i) It contains all modules $\chcFGB(M)$, where $M$ is in $\cO_\alg$. 

\vskip 4pt

(ii) It is closed under taking subquotients and extensions.

\vskip 8pt

For $\lambda\in\frt_{\text{alg}}^*$, we define $\OGBlambda$ to be the full subcategory of $\OGB$ consisting of modules $M$ which lie in $\DGmodlambda$.
\end{dfn}

We recall the following result from \cite[2.3.4]{JLS} which we will use in \ref{closed under tensoring}.
\begin{lemma}\label{Commuting lemma}
Let $L$ be a finite-dimensional locally $F$-analytic representation of $G$ over $E$. Consider a $\DgP$-module $M$. Then we have the isomorphism of $D(G)$-modules
\[L\otimes_E(D(G)\ot_{\DgP} M)\lra D(G)\otimes_{\DgP}(L\ot_E M)\]
\end{lemma}

\begin{lemma}\label{closed under tensoring}
Let $L \in \cO^\frg_\alg$ be finite-dimensional. Then $L$ lifts canonically to an algebraic representation of $\bG$ which we also denote by $L$.  For any $N \in \OGB$, the module $L\otimes_E N$ also lies in $\OGB$.
\end{lemma}
\begin{proof}
Our result will be a consequence of three claims that we will prove. Our first claim is that any object of the form $L\ot_E \chcFGB(M)$ lies in $\OGB$.
\[\begin{array}{rcl}
L\ot_E \left(D(G)\ot_{\DgP} M\right) & 
\stackrel{\ref{Commuting lemma}}{\cong}& D(G)\ot_{\DgP} \left(L\ot_E M\right)\\
\end{array}\]
where $L \ot_E M\in \cO_\alg$ by \cite[Thm 1.1(d)]{Hu}. This completes the proof of our first claim.
\vskip8pt

Our second claim is the following: if an object $N$ of $\OGB$ which has the property that $L \ot_E N$ is again in $\OGB$, then the same is true for any subquotient $N'$ of $N$. Write $N' =  N_1/N_2$ with submodules $N_2\subset N_1\subset N$. Then we have the isomorphism of $D(G)$-modules
\[L \ot_E N' =  L\ot_E \left(N_1/N_2\right)\cong (L\ot_E N_1)/(L\ot_E N_2) \,.\]
Now $(L\ot_E N_1)/(L\ot_E N_2)$ is a subquotient of $L \ot_E N$ which lies in $\OGB$. The module $L\ot_E N'$ lies then in $\OGB$ since $\OGB$ is closed under taking subquotients.

\vskip8pt

Now consider a $D(G)$-module $N$ which is part of an exact sequence of the following form
\begin{numequation}
0\lra N_1\lra N \lra N_2\lra 0
\end{numequation}
where $N_1$ and $N_2$ are objects in $\OGB$ which have the property that $L \ot_E N_1$ and $L \ot_E N_2$ are in $\OGB$. Our third claim is that $L\ot_E N$ lies in $\OGB$. The functor $L\ot_E (-)$ is an exact endofunctor on $\DGmod$. So we have the exact sequence of $D(G)$-modules
\begin{numequation}
0\lra L\ot_E N_1\lra L\ot_E N\lra L\ot_E N_2\lra 0 \;.
\end{numequation}
Since $\OGB$ is closed under taking extensions, it follows that $L\ot_E N$ lies in $\OGB$.
\end{proof}

\vskip 8pt
\begin{theorem}\label{categirical equivalence}
Let $\lambda,\mu\in \frt_E^*$ satisfy the following conditions:
\begin{itemize}
    \item[(i)] $\lambda$ and $\mu$ are compatible. 
    \item[(ii)] $\lambda$ and $\mu$ are both anti-dominant.
    \item[(iii)] $W^\circ_\lambda = W^\circ_\mu$.
\end{itemize}
\vskip8pt

Then the functors 
\[\cTlamu: \OGBlambda\lra \OGBmu\]
and 
\[\cTmula: \OGBmu\lra \OGBlambda\]
\vskip 8pt
induce an equivalence of categories. One has natural isomorphisms $\cTlamu \circ \cTmula \cong \id$ and $\cTmula \circ \cTlamu \cong \id$.
\end{theorem}
\begin{proof}
First we claim that $\cTlamu$ preserves the category $\OGB$. The functor $\cTlamu$ is a composition of projections and tensoring with $L(\onu)$. By \ref{closed under tensoring}, the category $\OGB$ is closed under taking tensor product with $L(\onu)$.  Since $\OGB$ is closed under taking submodules, it is preserved by projection functors.  Hence $\cTlamu$ preserves $\OGB$. 
\vskip 8pt

Now $\OGBlambda$ is a subcategory of $\DGmodlambda$. That $\OGBlambda$ is mapped to $\OGBmu$ by $\cTlamu$ follows from the fact that $\cTlamu$ maps $\DGmodlambda$ to $\DGmodmu$, cf. \cite[3.2.1]{JLS}. Similarly, $\cTmula$ maps $\OGBmu$ to $\OGBlambda$. That $\cTlamu \circ \cTmula \cong \id$ and $\cTlamu \circ \cTlamu \cong \id$ and that these functors induce equivalence of categories follows from \cite[3.2.1]{JLS}.
\end{proof}

\vskip 8pt
\begin{theorem}\label{Image of Oalglambda}
The image of $\Oalglambda$ under $\chcFGB$ is contained in $\OGBlambda$.
\end{theorem}
\begin{proof}
It is clear that the image of $\Oalglambda$ under $\chcFGB$ is contained in $\OGB$, by definition of $\OGB$. So it is enough to prove that the image is contained in $\DGmodlambda$. For $M\in \Oalglambda$ we have  
\[\chcFGB(M) = D(G) \ot_{\DgB} M = D(G) \ot_{\DgB} \pr_{|\lambda|}(M) = \pr_{|\lambda|}\bigg(D(G) \ot_{\DgB} M\bigg)\]
where the last equality holds by \cite[2.2.2 (iv)]{JLS}. Thus $\chcFGB(M)$ lies in $\DGmodlambda$.
\end{proof}

\begin{theorem}\label{main result v2}
The induced map $\chcFGB:K_0(\Oalglambda)\ra K_0(\OGBlambda)$ is injective.
\end{theorem}
\begin{proof}
The homomorphism of the statemet is the upper horizontal arrow in the commutative diagram 
\[
\begin{tikzcd}
K_0(\Oalglambda) \arrow[d] \arrow[rr] & & K_0\left(\OGBlambda\right)\arrow[d]\\
K_0(\cO_\alg) \arrow[rr] & & K_0\left(\DGcoadfl\right)
\end{tikzcd}
\]
\vskip 8pt
where the lower horizontal arrow is injective by \ref{our main result}. The vertical arrow on the left is injective, because $\cO_\alg = \bigoplus_{\lambda \in \frt^*_\alg/W} \Oalglambda$, where the direct sum is over the orbits of $W$ for the dot-action on $\frt^*_\alg$. Hence the composition of these maps is injective, which implies that the upper horizontal arrow is injective.
\end{proof}

\vskip 8pt
\begin{theorem}\label{commutative theorem}
Let $\lambda, \mu\in \frt_E^*$ be such that $\mu-\lambda$ is algebraic. Then the functors $\chcFGB$, $\Tlamu$ and  $\cTlamu$ give rise to the following commutative diagram of abelian groups.
\[
\begin{tikzcd}
K_0(\Oalglambda) \arrow[d, "\Tlamu"] \arrow[rr, "\chcFGB", hook] & & K_0\left(\OGBlambda\right)\arrow[d, "\cTlamu"]\\
K_0(\Oalgmu) \arrow[rr, "\chcFGB", hook] & & K_0\left(\OGBmu\right)
\end{tikzcd}
\]
\vskip 8pt
where the horizontal maps are injective. Additionally, if $\lambda,\mu$ satisfy the conditions (ii) and (iii) in \ref{categirical equivalence}, then the vertical maps are isomorphisms.
\end{theorem}
\begin{proof}
By \ref{Image of Oalglambda}, the images of $K_0(\Oalglambda)$ and $K_0(\Oalgmu)$ under $\chcFGB$ are contained in $K_0\left(\OGBlambda\right)$ and $K_0\left(\OGBmu\right)$ respectively. The commutativity of the diagram follows from 
\cite[4.1.12]{JLS}. The injectivity of the horizontal arrows follows from \ref{main result v2}. Now assume that $\lambda$ and $\mu$ satisfy the conditions (ii) and (iii) in \ref{categirical equivalence}. Then the left vertical map is an isomorphism by \cite[Propn. 7.8]{Hu}. The right vertical map is an isomorphism by \ref{categirical equivalence}.  
\end{proof}

\vskip8pt
\begin{example}
Let $\bf{G}=\bf{SL}_{\bf{2}}$ and $G=SL_2(F)$. Then $\frg=\frs\frl_2(F)$. Let $\frb\subset \frg$ be a Borel subalgebra, $\Delta=\Phi^+=\{2\rho\}$. Then $\frt_\alg^*=\bbZ\cdot \rho$. A weight $\lambda=c\cdot \rho$ where $c\in \bbZ$,  is dot-regular if and only if $c\neq 1$.

\vskip8pt

Let $\lambda=-\rho$ and consider $\cO_{\alg, |\lambda|}=\cO_{\alg, |-\rho|}$. The category $\cO_{\alg, |-\rho|}$ is a semisimple category and every object is isomorphic to a finite direct sum $M(-\rho)=L(-\rho)$. In this case $\chcFGB(M(-\rho))$ is a topologically simple $D(G)$-module. Thus $K_0(\OGBlambda)$ is a free $\bbZ$-module of rank $1$, generated by $[\chcFGB(M(-\rho))]$. In this case the map
\[\chcFGB:K_0(\cO_{\alg, |\lambda|})\lra K_0(\OGBlambda)\]

\vskip 4pt

is surjective, since both the domain and codomain are of rank $1$. Thus $\chcFGB$ is an isomorphism.

\vskip8pt

Now consider a weight $\lambda=c\cdot \rho$, where $c< -1$. Then $\lambda$ is antidominant and $M(\lambda)\cong L(\lambda)$. The dot-orbit of $\lambda$ is $W\cdot \lambda=\{\lambda, -\lambda-2\rho\}.$ Now $-\lambda-2\rho$ is dominant so that $L(-\lambda-2\rho)$ is finite dimensional. In this case the Grothendieck group $K_0(\cO_{\alg, |\lambda|})$ is given by 
\[K_0(\cO_{\alg, |\lambda|})=\bbZ[M(\lambda)]\oplus \bbZ[L(-\lambda-2\rho)].\]
Now $\chcFGB(M(\lambda))$ is topologically simple as a $D(G)$-module. We have the following exact sequence of smooth $G$-representations
\[0\lra \textbf{1}\lra \ind^G_B(\textbf{1})\lra \text{St}\lra 0\]
\vskip4pt

where $\text{St}$ denotes the Steinberg representation and $\textbf{1}$ denotes the trivial $1$-dimensional representation and $\ind^G_B(\textbf{1})$ denotes the smooth induction from $B$ to $G$.  Taking duals, we get the exact sequence of $D(G)$-modules
\[0\lra \text{St}'\lra \ind^G_B(\textbf{1})'\lra \textbf{1}\lra 0.\]

\vskip 4pt
Tensoring with $L(-\lambda-2\rho)$ we get the following exact sequence of $D(G)$-modules
\begin{numequation}\label{eqn in example}
0\lra L(-\lambda-2\rho)\ot_E\text{St}'\lra L(-\lambda-2\rho)\ot_E\ind^G_B(\textbf{1})'\lra L(-\lambda-2\rho)\lra 0.
\end{numequation}

\vskip 4pt
Now $\chcFGB(L(-\lambda-2\rho))\cong L(-
\lambda-2\rho)\ot_E \ind^G_B(\textbf{1})'$. Thus from \ref{eqn in example} it follows that

\[[\chcFGB(L(-\lambda-2\rho))]= [L(-\lambda-2\rho)] + [L(-\lambda-2\rho)\ot_E \text{St}']\]

\vskip 4pt
Furthermore, $L(-\lambda-2\rho)$ is absolutely irreducible, and hence $\End_{E[\bS\bL_2]}\Big(L(-\lambda-2\rho)\Big) = E$. This implies that $L(-\lambda-2\rho)\ot_E\text{St}'$ is irreducible by \cite[4.2.8]{Emerton_memoire}. It follows that

\[K_0(\OGBlambda)=\bbZ[\chcFGB(M(\lambda))]\text{ }\oplus \text{ }\bbZ[L(-\lambda-2\rho)] \text{ } \oplus\text{ }\bbZ[L(-\lambda-2\rho)\ot_E\text{St}'].\]
\vskip 4pt
Hence the image of the map
\[\chcFGB:K_0(\cO_{\alg, |\lambda|})\lra K_0(\OGBlambda)\]
has codimension $1$.
\end{example}

\vskip 4pt

\bibliographystyle{abbrv}
\bibliography{mybib}

\end{document}